\documentclass[12pt]{amsart}
\usepackage[utf8]{inputenc}
\usepackage[english]{babel}
\usepackage{a4wide}
\usepackage{amssymb}
\usepackage[hidelinks]{hyperref}
\usepackage{todonotes}
\usepackage[dvipsnames]{xcolor}

\usepackage{tikz}
\tikzstyle{edge}=[shorten <=2pt, shorten >=2pt, >=stealth, line width=1.5pt]
\tikzstyle{vertex}=[circle, fill=white, draw, minimum size=6pt, inner sep=0pt,
                    outer sep=0pt]
\usetikzlibrary{positioning, 
                quotes}

\usepackage[normalem]{ulem}

\theoremstyle{definition}
\newtheorem{definition}{Definition}[section]
\newtheorem{example}{Example}
\newtheorem*{remark*}{Remark}
\newtheorem*{claim*}{Claim}
\theoremstyle{remark}

\theoremstyle{plain}
\newtheorem{theorem}{Theorem}[section]

\newtheorem{prop}[theorem]{Proposition}
\newtheorem{observation}[theorem]{Observation}
\newtheorem{lemma}[theorem]{Lemma} 

\newtheorem{question}[theorem]{Question}

\newtheorem{fact}[theorem]{Fact}

\theoremstyle{definition}
\newtheorem{claim}[theorem]{Claim}

\def\Aut{\mathop{\mathrm{Aut}}\nolimits}

\def\ra#1#2{{#1}_{#2}}
\def\id{\mathrm{id}}
\def\C{\mathbb{C}_3}
\def\oC{\overbar{\C}}

\def\Z{\mathbb Z}
\def\soc{\sigma_{\oC}}
\def\NP{NP}

\def\cyc{\mathbf{C}}

\newcommand{\overbar}[1]{\mkern 1.5mu\overline{\mkern-1.5mu#1\mkern-1.5mu}\mkern 1.5mu}

\def\bA{\mathbf{A}}

\def\fB{\mathfrak{B}}
\DeclareMathOperator {\NSP}{NSP}

\title[Circular Chromatic Numbers, Signability, Relation Algebras]{Circular Chromatic Numbers, Signability, Relation Algebras, and Network Satisfaction Problems}
\author[M. Bodirsky]{Manuel Bodirsky}
\address{Institut für Algebra, TU Dresden, Germany}
\email{manuel.bodirsky@tu-dresden.de}
\author[S. Guzmán-Pro]{Santiago Guzmán-Pro}
\address{Institut für Algebra, TU Dresden, Germany}
\email{santiago.guzman\_pro@tu-dresden.de }
\author[M. Jahn]{Moritz Jahn}
\address{Institut für Algebra, TU Dresden, Germany}
\author[M. Konečný]{Matěj Konečný}
\address{Institut für Algebra, TU Dresden, Germany}
\email{matej.konecny@tu-dresden.de}
\author[P. Winkler]{Paul Winkler}
\address{Institut für Algebra, TU Dresden, Germany}
\email{paul.winkler1@tu-dresden.de}

\thanks{Manuel Bodirsky, Santiago Guzmán-Pro, Mat\v{e}j Kone\v{c}n\'y, and Paul Winkler received funding from the ERC (Grant Agreement no. 101071674, POCOCOP). Views and opinions expressed are however
those of the authors only and do not necessarily reflect those of the European Union or the European Research
Council Executive Agency.}

\begin{document}
\bibliographystyle{alpha}

\begin{abstract}
In this paper, we characterize finite graphs with circular chromatic number  less than 3 in terms of
the existence of certain signings ($\Z_2$-labellings studied in the context of signed graphs). In fact, we construct a signed graph 
which is universal for all such signings -- called anti-triangle-signings in this paper -- of finite $\overbar{K_3}$-free graphs, and is closely related to 
the generic circular triangle-free graph studied by Bodirsky and Guzm\'an-Pro. Moreover, our universal structure gives
rise to a representation of the relation algebra $\ra{56}{65}$. We then use this representation to show that the network
satisfaction problem described by this relation algebra belongs to NP.  This concludes the full classification of the
existence of a universal square representation, as well as the
complexity of the corresponding network satisfaction problem, for relation algebras with at most four atoms.
\end{abstract}

\maketitle

\section{Introduction}
In a recent paper~\cite{Bodirsky2025RA,Bodirsky2025RALIPICS}, Bodirsky, Jahn, Kn\"auer, Kone\v{c}n\'y, and Winkler study 
relation algebras with at most 4 atoms, their representations, and the corresponding network satisfaction problems.\footnote{See Sections~\ref{sec:ra} and~\ref{sec:ra56} for definitions of relation algebras.} 
In a single case out of more than a hundred (the relation algebra called $\ra{56}{65}$), 
the existence of a universal square representation and containment of the network satisfaction problem in \NP{} were left open. 
This gap was the starting point for the present work. 

Relation algebras describe certain structures with binary relations.
In the case of the relation algebra $\ra{56}{65}$, these structures can be seen as signed graphs, that is, graphs together with a \emph{signing} (i.e., a $\Z_2$-labelling) of the edges.\footnote{See Sections~\ref{sec:switching} and~\ref{sec:circular} for graph theoretic definitions.} 
While trying to better understand the relation algebra $\ra{56}{65}$, 
we computed minimal problematic graphs connected to it. 
In an unexpected turn of events, these turned out to 
correspond exactly to the forbidden subgraphs describing the generic circular triangle-free graph $\C$ introduced earlier by Bodirsky and Guzm\'an-Pro~\cite{Bodirsky2024Circular}. This led us to construct a special signing of the edges of the complement of $\C$, and from this construction, the desired relation algebra results follow easily.

Consequently, this paper contains results about both $\C$ and $\ra{56}{65}$. In particular, we prove the following.
\begin{theorem}\label{thm:c_balanceable}
A triangle-free graph $G$ embeds into $\C$ if and only if its complement $\overbar{G}$ is anti-triangle-signable.
\end{theorem}

Theorem~\ref{thm:c_balanceable} has the following surprising corollary, which adds yet another characterization of graphs with circular chromatic number less than 3, see also~\cite{Brandt1999, Bodirsky2024Circular, guzmanAMC438}.
\begin{theorem}\label{cor:balanceable}
A finite graph has circular chromatic number less than 3 if and only if it is a spanning subgraph of a finite triangle-free graph whose complement is anti-triangle-signable.
\end{theorem}

We then refine Theorem~\ref{thm:c_balanceable} and obtain a particularly nice signing of $\oC$.

\begin{theorem}\label{thm:main}
There is an anti-triangle-signing $\soc$ of $\oC$ such that for every finite $\overbar{K_3}$-free graph $G$ with an anti-triangle-signing $\sigma_G$ there is a sign-preserving embedding $f\colon (G,\sigma_G)\to (\oC,\soc)$.
\end{theorem}

In turn, Theorem~\ref{thm:main} is the key ingredient for our results on the relation algebra side, where we prove the following theorem, thereby answering~\cite[Question~4.29]{Bodirsky2025RAv1}.

\begin{theorem}\label{thm:main_ra}
The relation algebra $\ra{56}{65}$ has a finitely bounded universal square representation. Consequently, the network satisfaction problem for the relation algebra $\ra{56}{65}$ is in \NP{}.
\end{theorem}

\section{Signings and switching}\label{sec:switching}
We follow standard terminology. In particular, a \emph{graph} $G$ is a set $V(G)$ of \emph{vertices} together with a set $E(G) \subseteq {{V(G)} \choose 2}$ of \emph{edges} (following standard notation, we will write $x_1x_2$ instead of $\{x_1,x_2\}$). All graphs in this paper will be finite or countably infinite. Given graphs $G$ and $H$, we say that $H$ is a \emph{subgraph of $G$} and write $H\subseteq G$ if $V(H) \subseteq V(G)$ and $E(H) \subseteq E(G)$. If $H\subseteq G$ and $E(H) = E(G) \cap {V(H) \choose 2}$, we say that $H$ is an \emph{induced subgraph of $G$}. 
The \emph{disjoint union} of $G$ and $H$ is the graph $G + H := (V(G) \cup V(H), E(G) \cup E(H))$ where
$V(G)\cap V(H) = \emptyset = E(H)\cap E(G)$. We say that a vertex $v$ is \emph{universal} in $G$ if
$v$ is adjacent to every vertex $u$ in $V(G)\setminus \{v\}$.

A \emph{signed graph} is a pair $(G,\sigma)$ where $G$ is a graph and $\sigma$ is a function $E(G) \to \Z_2$, which is called a \emph{signing of $G$}. Signed graphs have been introduced by Harary in 1953~\cite{Harary1953} and have connections to other areas of mathematics and to applications (e.g. knot theory, sociology, chemistry, or statistical physics, see for example~\cite{Zaslavsky2018}).

Given a finite subgraph $H\subseteq G$, we write $\sigma(H) = \sum_{e\in E(H)} \sigma(e)$ (with addition in $\Z_2$). Slightly abusing notation, if $\mathcal H$ is some set of subgraphs of $G$, we can thus consider $\sigma|_{\mathcal H}$ to be the function with domain $\mathcal H$ sending every $H\in \mathcal H$ to $\sigma(H)$.
Denote by $\cyc(G)$ the collection of all chordless cycles of $G$. If $\beta\colon \cyc(G) \to \Z_2$ is a function, we say that a signing $\sigma$ of $G$ is a \emph{$\beta$-signing} if $\sigma|_{\cyc(G)} = \beta$, and we say that a graph $G$ is \emph{$\beta$-signable} if there exists a $\beta$-signing of $G$.

A graph $G$ is called \emph{triangle-signable}\footnote{These graphs were originally studied by Conforti, Cornu{\'e}jols, Kapoor, and Vu{\v{s}}kovi{\'{c}}~\cite{Conforti1995,Conforti2000} under the name \emph{even-signable}. Following a suggestion of T. Zaslavsky, we opt for the term \emph{triangle-signable} as it more accurately reflects the distinct role of triangles in the definition.} if it is $\beta$-signable for the function $\beta \colon \cyc(G) \to \Z_2$ given by $\beta(C) := 1$ if $\lvert C\rvert = 3$ and $\beta(C) := 0$ otherwise. Triangle-signable graphs generalise graphs with no odd holes (i.e., induced cycles of odd length at least 5). Similarly, a graph $G$ is called \emph{odd-signable} if it is $\beta$-signable for $\beta$ being constant 1. Odd-signable graphs generalise even-hole-free graphs, and Chudnovsky, Kawarabayashi, and Seymour~\cite{Chudnovsky2005} provided a polynomial-time recognition algorithm for them.

In this paper we will be interested in a related variant which we call \emph{anti-triangle-signable} graphs. These are graphs that are $\beta$-signable for the function $\beta \colon \cyc(G) \to \Z_2$ given by $\beta(C) := 0$ if $\lvert C\rvert = 3$ and $\beta(C) := 1$ otherwise. Such a $\beta$-signing $\sigma$ of $G$ is then called an \emph{anti-triangle-signing}.

\medskip

In 1982, Truemper proved Theorem~\ref{thm:truemper} below;
we are following the presentation from an alternative proof of Truemper's theorem by Conforti, Gerards, and Kapoor~\cite{Conforti2000}.
For stating the theorem, we need the following definitions. A \emph{wheel} is a chordless cycle $C$ plus a vertex connected to at least three elements of $C$.

A \emph{three path configuration} is a graph $P$ for which there are paths $P_1$, $P_2$, and $P_3$ with endpoints $x_1$, $y_1$, $x_2$, $y_2$, and $x_3$, $y_3$, respectively, such that the following is true:
\begin{itemize}
    \item $V(P) = V(P_1) \cup V(P_2) \cup V(P_3)$,
    \item for every $i\neq j\in \{1,2,3\}$ it holds that $V(P_i)\cap V(P_j) \subseteq \{x_i,x_j,y_i,y_j\}$,
    \item $E(P) = E(P_1) \cup E(P_2) \cup E(P_3) \cup E$ for some set $E$, and
    \item exactly one of the following is true:
    \begin{enumerate}
        \item $x_1 = x_2 = x_3$, $y_1 = y_2 = y_3$, $\lvert V(P_i)\rvert \geq 3$ for every $i\in \{1,2,3\}$, and $E = \emptyset$.
        \item $y_1 = y_2 = y_3$, $E = \{x_1x_2,x_1x_3,x_2x_3\}$, and $V(P_i) \cap V(P_j) = \{y_1\}$ for all distinct $i,j \in \{1,2,3\}$.
        \item $V(P_i) \cap V(P_j) = \emptyset$ for all distinct $i,j\in \{1,2,3\}$ and $$E = \{x_1x_2,x_1x_3,x_2x_3,y_1y_2,y_1y_3,y_2y_3\}.$$
    \end{enumerate}
\end{itemize}

\begin{theorem}[Truemper~\cite{Truemper1982}]\label{thm:truemper}
Let $G$ be a graph and $\beta \colon \cyc(G) \to {\mathbb Z}_2$ be a function.
Then $G$ is $\beta$-signable if and only if every induced subgraph $H\subseteq G$ isomorphic to a 3-path configuration or to a wheel is $\beta|_{\cyc(H)}$-signable.
\end{theorem}

\begin{example}\label{ex:C6-complement}
 One of the simplest three path configurations is $\overbar{C_6}$ (the complement of $C_6$, see Figure~\ref{fig:c6}).
 We argue that $\overbar{C_6}$ is not anti-triangle-signable, and so, according to Theorem~\ref{thm:truemper},
 every anti-triangle-signable graph $G$ is $\overbar{C_6}$-free. Proceeding by contradiction, suppose there
 is such a signing and let $x,y,z,a,b,c,e,f,g\in\mathbb Z_2$ denote the signs; see Figure~\ref{fig:c6}.
 Since the signs of the top triangle must add up
 to zero, at least one of these values must be zero. Up to symmetries of $\overline{C_6}$, we may assume
 that $z = 0$, and so $x = y$. Since the signs of all chordless $4$-cycles must add up to $1$, we obtain the following
 equations 
    \[
    a + b+ e = 1-x = 1-y = b + c + f \text{ and so, } a + e -f = c.
    \]
    Using again that triangles add up to zero we know that $e + g = -f$, and substituting
    in the last equation above, we get that $a + 2e + g = c$ which yields
    $a + g = c$. Finally, this implies that $a + g + c + z = 2c + z = 0$, contradicting
    that all $4$-cycles add up to $1$.
    \begin{figure}
    \begin{tikzpicture}[every edge quotes/.style = {auto, font=\footnotesize, inner sep=1.2pt},scale = 0.8]
    \begin{scope}[xshift=8cm]
    \node [vertex] (2) at (0,1) {};
    \node [vertex] (1) at (-1.25,1.7) {};
    \node [vertex] (3) at (1.25,1.7) {};

    \node [vertex] (22) at (0,-1) {};
    \node [vertex] (11) at (-1.25,-1.7) {};
    \node [vertex] (33) at (1.25,-1.7) {};
    
   \draw   (2)  edge["$x$"] (1)
        (3)  edge["$y$"] (2)
        (1)  edge["$z$"] (3)
        (1)  edge["$a$"] (11)
        (2)  edge["$b$"] (22)
        (3)  edge["$c$"] (33)       
        (11)  edge["$e$"] (22)
        (22)  edge["$f$"] (33)
        (33)  edge["$g$"] (11);
  \end{scope}
  \end{tikzpicture}
  \caption{The three path configuration $\overbar{C_6}$}
  \label{fig:c6}
  \end{figure}
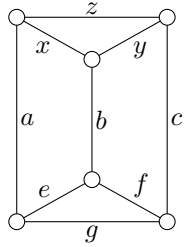

\end{example}
Next, we introduce a key concept in the area of signed graphs which will also play an important role in this paper.
\begin{definition}
Given a graph $G$ with a signing $\sigma$ and a set $S\subseteq V(G)$, the \emph{switch of $\sigma$ over $S$}
is the signing $\sigma^S$ defined by $\sigma^S(e) := \sigma(e)+1$ if $\lvert S\cap e\rvert = 1$ and
$\sigma^S(e) := \sigma(e)$ otherwise. (In other words, we switch the signs of all edges with exactly one endpoint in $S$.) We also write $\sigma^v$ for $\sigma^{\{v\}}$.
\end{definition}
\begin{observation}\label{obs:switch_commutative} 
For every graph $G$ with a signing $\sigma$ and a pair of subsets $S,T\subseteq V(G)$ it holds that
$$\left(\sigma^S\right)^T = \left(\sigma^T\right)^S = \sigma^{S\triangle T},$$
where $S\triangle T$ is the symmetric difference of $S$ and $T$.\qed
\end{observation}

\begin{definition}
Given a graph $G$ with signings $\sigma$ and $\sigma'$, we say that $\sigma$ and $\sigma'$ are \emph{switching equivalent} if there is some $S\subseteq V(G)$ with $\sigma' = \sigma^S$.
\end{definition}

We will need the following well-known results:\footnote{Technically, \cite{zaslavskyDAM4} considers also signs of non-chordless cycles, but it is easy to see that these are uniquely determined by $\cyc(G)$.}
\begin{lemma}[Lemma~3.1 in~\cite{zaslavskyDAM4}]\label{lem:tree_unique_extension}
Let $G$ be a graph, let $\beta\colon \cyc(G) \to \Z_2$ be some function such that $G$ is $\beta$-signable, and let $T$ be a maximal acyclic (not necessarily induced) subgraph of $G$. Then for every signing $\sigma_T \colon E(T) \to \Z_2$ there is a unique $\beta$-signing $\sigma$ of $G$ such that $\sigma_T = \sigma|_{E(T)}$.
\end{lemma}
\begin{theorem}[Proposition 3.2 in~\cite{zaslavskyDAM4}]\label{thm:switching_equivalent}
Let $G$ be a graph and let $\sigma$ and $\sigma'$ be signings of $G$. Then $\sigma$ and $\sigma'$ are switching equivalent if and only if $\sigma|_{\cyc(G)} = \sigma'|_{\cyc(G)}$.
\end{theorem}

\section{The generic circular triangle-free graph}\label{sec:circular}
In~\cite{Bodirsky2024Circular}, Bodirsky and Guzmán-Pro introduced the \emph{generic circular triangle-free graph} $\C$,
whose vertices are unit complex numbers with rational argument\footnote{In this paper, the range of $\arg$ is $(-\pi,\pi]$.}
such that $(z,w)$ is an edge of $\C$ if and only if $\lvert\arg(z/w)\rvert > \frac{2}{3}\pi$. This graph has many
interesting properties; for example, a finite graph $G$ has circular chromatic number less than 3 if and only if it
admits a homomorphism to $\C$; and a graph $G$ of independence number at most 2 is a
unit Helly circular-arc graph if and only if it embeds into $\overbar{\C}$. The authors also described $\C$ by
means of forbidden subgraphs; we restate this description in Theorem~\ref{thm:santiago} below.
Recall that $G+H$ denotes the disjoint union of $G$ and $H$, and we write $2G$ as a shortcut for $G+G$.

\begin{theorem}[Theorem~12 and Corollary~43 from~\cite{Bodirsky2024Circular}]\label{thm:santiago}
A graph $G$ embeds into $\C$ if and only if it is a $\{K_3, 2K_2+K_1, C_5+K_1, C_6\}$-free
graph.
\end{theorem}

Theorem~\ref{thm:c_balanceable} is an easy corollary of Theorems~\ref{thm:truemper} and~\ref{thm:santiago}.
Let $W_4$ and $W_5$ be the graphs obtained by adding a universal vertex to $C_4$ and to $C_5$, respectively.
Notice that both $W_4$ and $W_5$ are wheels, and that $W_4 = \overbar{2K_2+K_1}$, and $W_5 = \overbar{C_5+K_1}$.

\begin{proof}[Proof of Theorem~\ref{thm:c_balanceable}]
To simplify notation, we will prove that a $\overbar{K_3}$-free graph $G$ embeds into $\overbar{\C}$ if and only if it is anti-triangle-signable.
By looking at the complements, Theorem~\ref{thm:santiago} tells us that  $G$ embeds into $\oC$ if and only if it has no independent set of size 3 and it is $\{W_4,W_5,\overbar{C_6}\}$-free. In Example~\ref{ex:C6-complement} we argue
that $\overbar{C_6}$ is a three path configuration that is not anti-triangle-signable, and with similar arguments one can
notice that $W_4$ and $W_5$ are wheels which are not anti-triangle-signable.
Thus, by Theorem~\ref{thm:truemper}, if $G$ is anti-triangle-signable, then it is $\{W_4,W_5,\overbar{C_6}\}$-free, and so
it indeed embeds into $\overbar{\C}$ by Theorem~\ref{thm:santiago}. 

In order to prove the other implication, let $G$ be a graph which embeds into $\oC$. By Theorem~\ref{thm:santiago}, $G$ is $\{\overbar{K_3},W_4,W_5,\overbar{C_6}\}$-free. By Theorem~\ref{thm:truemper}, it suffices to prove that every three path configuration and every wheel in $G$ is anti-triangle-signable.

First, let $P \subseteq G$ induce a three path configuration as witnessed by paths $P_1, P_2, P_3$ such that $P_i$ has endpoints $x_i,y_i$ for every $i\in \{1,2,3\}$. Note that there is $i\in \{1,2,3\}$ such that $\lvert V(P_i)\rvert = 2$, as otherwise the internal vertices of the paths contain an independent set of size 3. Without loss of generality, we may assume that $\lvert V(P_1)\rvert = 2$.
It follows that $|\{x_1,x_2,x_3\}| = 3$. 
We now consider the following case distinction:
\begin{enumerate}
    \item If $\lvert V(P_2)\rvert = \lvert V(P_3)\rvert = 2$, then either $y_1 = y_2 = y_3$, and so $P$ is isomorphic to $K_4$,
    or  $|\{y_1,y_2,y_3\}| = 3$, and so $P$ is isomorphic to $\overbar{C_6}$. Since $G$ is $\overbar{C_6}$-free, it must
    be the case that $P$ induces a $4$-clique, which is anti-triangle-signable by labelling all edges with $0$.
    \item\label{pf:case:2} If $\lvert V(P_2)\rvert = 2$ and $\lvert V(P_3)\rvert \geq 3$, then it cannot be the case that
    $|\{y_1,y_2,y_3\}| = 3$ because $x_1$ and $y_2$ together with any internal vertex of
    $P_3$ would induce an independent set of size $3$, contradicting the choice of $G$. Hence,
    $y_1=y_2=y_3$, and in this case, we can label one arbitrary edge of $P_3$ by $1$
    and all other edges by $0$ to obtain an anti-triangle-signing. 
    \item The case when $\lvert V(P_3)\rvert = 2$ and $\lvert V(P_2)\rvert \geq 3$ is symmetric to the one above.
    \item The above are the only possible cases when $G$ is a $\overbar{K_3}$-free graph. Indeed, in any case
    not considered above, we have that  
    $\lvert V(P_2)\rvert,\lvert V(P_3)\rvert \geq 3$, and so $x_1$ together with any internal vertex $z_2$ of $P_2$ and any internal
    vertex $z_3$ of $P_3$ induce an independent set of size 3.
    \end{enumerate}

Next, observe that $C_n$ contains an independent set of size 3 if and only if $n\geq 6$, hence $G$ contains no chordless cycles of length at least $6$. Clearly, the only wheel over $C_3$ is $K_4$ and it is anti-triangle-signable as noted above. 
Since $G$ does not embed $W_4$, the only possible wheel over $C_4$ is $C_4$ with an extra vertex $v$ connected to
three vertices of the $C_4$. Notice that this induces a three path configuration as in case (\ref{pf:case:2}) above
with $y_1 = y_2 = y_3$, and so it is anti-triangle-signable. Finally,  
as $G$ does not embed $W_5$, there are three possible wheels over $C_5$ which $G$ might contain. One of 
them contains $\overline{K_3}$ (so it cannot embed into $G$), and the other two admit
an anti-triangle-signing (see Figure~\ref{fig:wheels5}). 
Hence $G$ is anti-triangle-signable by Theorem~\ref{thm:truemper}.

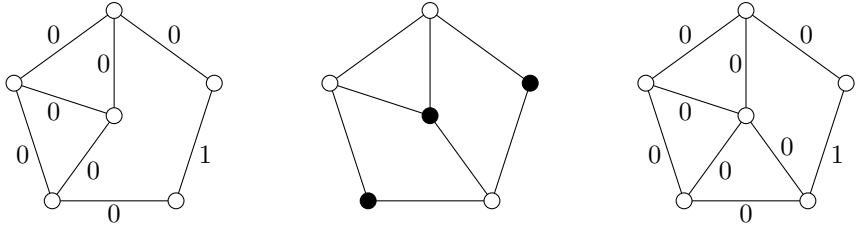
\begin{figure}
    \begin{tikzpicture}[every edge quotes/.style = {auto, font=\footnotesize, inner sep=1.2pt},scale = 0.7]
    \begin{scope}[]
    \node [vertex] (a) at (90:2) {};
    \node [vertex] (b) at (162:2) {};
    \node [vertex] (c) at (234:2) {};
    \node [vertex] (d) at (306:2) {};
    \node [vertex] (e) at (18:2) {};
    \node [vertex] (v) at (0:0) {};
      
   \draw   (a)  edge["$0$"] (e);
   \draw   (e)  edge["$1$"] (d);
   \draw   (d)  edge["$0$"] (c);
   \draw   (c)  edge["$0$"] (b);
   \draw   (b)  edge["$0$"] (a);
   \draw   (v)  edge["$0$"] (c);
   \draw   (v)  edge["$0$"] (b);
   \draw   (v)  edge["$0$"] (a);
  \end{scope}

  \begin{scope}[xshift = 6cm]
    \node [vertex] (a) at (90:2) {};
    \node [vertex] (b) at (162:2) {};
    \node [vertex, fill = black] (c) at (234:2) {};
    \node [vertex] (d) at (306:2) {};
    \node [vertex, fill = black] (e) at (18:2) {};
    \node [vertex, fill = black] (v) at (0:0) {};

    \draw   (a)  edge (e);
   \draw   (e)  edge (d);
   \draw   (d)  edge (c);
   \draw   (c)  edge  (b);
   \draw   (b)  edge  (a);
   \draw   (v)  edge (d);
   \draw   (v)  edge (b);
   \draw   (v)  edge (a);
  \end{scope}
  
  \begin{scope}[xshift = 12cm]
    \node [vertex] (a) at (90:2) {};
    \node [vertex] (b) at (162:2) {};
    \node [vertex] (c) at (234:2) {};
    \node [vertex] (d) at (306:2) {};
    \node [vertex] (e) at (18:2) {};
    \node [vertex] (v) at (0:0) {};
      
   \draw   (a)  edge["$0$"] (e);
   \draw   (e)  edge["$1$"] (d);
   \draw   (d)  edge["$0$"] (c);
   \draw   (c)  edge["$0$"] (b);
   \draw   (b)  edge["$0$"] (a);
   \draw   (v)  edge["$0$"] (c);
   \draw   (v)  edge["$0$"] (b);
   \draw   (v)  edge["$0$"] (a);
   \draw   (v)  edge["$0$"] (d);
  \end{scope}
  \end{tikzpicture}

\caption{Three non-isomorphic wheels over $C_5$ different from $W_5$.}
 \label{fig:wheels5}
\end{figure}
\end{proof}

Now we can also prove Theorem~\ref{cor:balanceable}. 

\begin{proof}[Proof of Theorem~\ref{cor:balanceable}]
By~\cite{Bodirsky2024Circular}, a graph $G$ has circular chromatic number less than 3 if and only if $G$ has a homomorphism to $\C$. It is easy to verify that $G$ has a homomorphism to $\C$ if and only if it has an injective homomorphism to $\C$. Clearly, $G$ is isomorphic to a spanning subgraph of its image under an injective homomorphism, and Theorem~\ref{thm:c_balanceable} then gives the conclusion.
\end{proof}

Next, we turn our attention to constructing a particularly nice anti-triangle-signing of $\overbar{\C}$ in order to prove Theorem~\ref{thm:main}.

\begin{definition}\label{defn:sigmac}
Fix a partition $V(\oC) = C_0\cup C_1$ into disjoint dense subsets, and a set $P \subseteq V(\oC)$ on which $\oC$ induces a 5-cycle.
Given $x\in V(\oC)$, let $i(x)\in \{0,1\}$ be the index such that $x\in C_{i(x)}$. Enumerate $P = \{p^0,\ldots,p^4\}$ such that $i(p^0) = i(p^1)$ and
$p^ip^{i+1} \in E(\oC)$ for every $i\in \{0,\ldots,4\}$ with addition modulo 5.
For every $x\in V(\oC)$, we define $p_x := p^i$ where $i\in \{0,\ldots,4\}$ is the smallest number with the property that $xp^i \in E(\oC)$;
such $p_x$ exists as $\oC$ is $\overbar{K_3}$-free, and in particular, $p_{p^0} = p^1$, $p_{p^4} = p^0$, and $p_{p^i} = p^{i-1}$ for $i\in\{1,2,3\}$.
Let $T$ be the spanning subtree of $\oC$ defined by $V(T) = V(\oC)$ and $E(T) = \{xp_x : x\in V(T)\}$. 
Define a signing $\sigma_T$ of $T$ by putting 
$\sigma_T(xp_x) = i(x)$, and use Lemma~\ref{lem:tree_unique_extension} to
get an anti-triangle-signing $\soc$ of $\oC$ such that $\sigma_T = \soc|_{E(T)}$.
\end{definition}
We will need the following two technical lemmas.

\begin{lemma}\label{lem:switch}
Let $X$ and $Y$ be induced subgraphs of $\oC$ and let  $f\colon X\to Y$ be an isomorphism 
such that for every $x\in V(X)$ and every $p\in P$ we have that $px \in E(\oC)$ if and only if $pf(x) \in E(\oC)$.
Then, for every edge $xy\in E(X)$ the following equality holds,
$$\soc(f(x)f(y)) = \soc(xy) + i(x) + i(f(x)) + i(y) + i(f(y)).$$
\end{lemma}
\begin{proof}
Pick an arbitrary edge $xy \in E(X)$. First note that if $x\in P$ then $f(x) = x$ and similarly for $y$ (this follows from the assumption on adjacencies to $P$). If $xy$ is an edge of $P$ then the conclusion follows trivially. So $xy$ is not an edge of $P$.

Next, observe that there is $p\in P$ such that both $xp$ and $yp$ are edges
of $\oC$:
if $x,y\notin P$, then each of $x$ and $y$ has at least three
neighbours in $P$, and hence they have a common neighbour in $P$.
Otherwise, say $x\in P$. Then $y\notin P$, and $y$ is adjacent to at
least one of the two neighbours of $x$ in $P$, since these two vertices
are non-adjacent and $\oC$ is $\overbar{K_3}$-free.
Consequently, $\soc(xy) + \soc(xp) + \soc(yp) = 0$, and
$\soc(f(x)f(y)) + \soc(f(x)p) + \soc(f(y)p) = 0$, as $\soc$ is an
anti-triangle-signing. Summing these equalities, we get
$$
\soc(f(x)f(y)) = \soc(xy) + \soc(xp) + \soc(yp)
+ \soc(f(x)p) + \soc(f(y)p).
$$

Next, 
if $x\in P$, then $f(x)=x$, and hence
$$
\soc(xp) + \soc(f(x)p) = i(x)+i(f(x))=0.
$$
Otherwise,
observe that the neighbours of $x$ in $P$ appear consecutively in the
cycle. Hence, since the signs of triangles add up to $0$, the sign of
any edge from $x$ to $P$ determines the signs of all edges from $x$ to
$P$.
Also, for every $x\in V(X)$ we have that $p_x=p_{f(x)}$ (this follows
from the definition of $p_x$ and the assumption that for every $p\in P$
we have that $px\in E(\oC)$ if and only if $pf(x)\in E(\oC)$).
Consequently, using the definition of $\soc$, we see that
$\soc(xp) = \soc(f(x)p)$ if and only if $i(x) = i(f(x))$.
Equivalently,
$$
\soc(xp) + \soc(f(x)p) = i(x) + i(f(x)).
$$
In either case, this equality holds for $x$, and similarly for $y$.
The conclusion is now immediate.
\end{proof}

\begin{lemma}\label{lem:cswitching}
Let $S, X, Y\subseteq V(\C)$ be finite sets such that $S\subseteq X$ and $S\cap P=\emptyset$. Then there is a map
$f\colon X\to V(\C)$ with the following properties:
\begin{enumerate}
    \item $f$ is the identity on $X\setminus S$, and
    \item $f$ is a sign-preserving isomorphism $(G,\soc^S|_{E(G)})\to(H,\soc|_{E(H)})$, where $G$ is the subgraph of $\oC$ induced by $X$ and $H$ is the subgraph induced by $f[X]$.
    \item $f[S]$ is disjoint from $P\cup Y$.
\end{enumerate}
\end{lemma}

\begin{proof}
First, note that from the definition of $\C$ we have the following:
\begin{claim}\label{cl:epsilon}
If $x,x',y,y'\in V(\C)$ are such that 
$$\max\left(\lvert \arg(x/x')\rvert, \lvert \arg(y/y')\rvert\right) < \frac{1}{2} \left\lvert \frac{2\pi}{3} - \lvert\arg(x/y)\rvert\right\rvert,$$ then $xy$ is an edge of $\C$ (respectively $\oC$) if and only if $x'y'$ is.
\end{claim}

Fix a finite set $Y\subseteq V(\C)$ and put
$$\varepsilon = \frac{1}{2}\min\left\{\left\lvert \frac{2\pi}{3} - \lvert\arg(x/y)\rvert\right\rvert : x,y\in X\cup P\right\}.$$
Note that $\varepsilon>0$.
Define an injective function $f\colon X \to V(\oC)$ as follows: If $x\in X\setminus S$, put $f(x)=x$. For $x\in S$, keep choosing $f(x)$ one-by-one so that $\lvert \arg(f(x)/x)\rvert < \varepsilon$, $f(x)\in C_{1-i(x)}$, $f(x)\notin P\cup Y$, and $f(x)$ is distinct from all previously chosen values of $f$.
By Claim~\ref{cl:epsilon}, $f$ is an isomorphism between the graphs
induced by $X$ and by $f[X]$. Moreover, 
for every $x\in X$ and every $p\in P$ we have  that $xp \in E(\oC)$ if and only if $f(x)p \in E(\oC)$. Consequently, Lemma~\ref{lem:switch}
tells us that for every $x,y\in X$ with $xy \in E(\oC)$ the equality 
$$\soc(f(x)f(y)) = \soc(xy) + i(x) + i(f(x)) + i(y) + i(f(y))$$
holds. Finally,  
by the definition of $f$ we know that $i(x) + i(f(x)) + i(y) + i(f(y)) = 1$ if and only if exactly one of $x$ and $y$ is in $S$. 
\end{proof}

Now we are ready to prove Theorem~\ref{thm:main}.
\begin{proof}[Proof of Theorem~\ref{thm:main}]
We will prove that the signing $\soc$ from Definition~\ref{defn:sigmac} has these properties. Let $C_0$, $C_1$, and $P$
be as in Definition~\ref{defn:sigmac}. We need to prove that given a finite $\overbar{K_3}$-free graph $G$ with an
anti-triangle-signing $\sigma_G$, there is a sign-preserving embedding $f\colon (G,\sigma_G)\to (\oC,\soc)$.
Fix an anti-triangle-signing $\sigma_G$ of such a graph $G$.

By Theorem~\ref{thm:c_balanceable}, applied to $\overbar{G}$, there is an embedding $g\colon G\to \overbar{\C}$. By composing $g$ with a sufficiently small rational rotation, we may
assume that $g[V(G)]\cap P=\emptyset$. Define a signing $\tau$ of $G$ by $\tau(xy)=\soc(g(x)g(y))$. By Theorem~\ref{thm:switching_equivalent}, there is $S\subseteq V(G)$ such that $\sigma_G=\tau^S$.
Applying Lemma~\ref{lem:cswitching} to $g[S]\subseteq g[V(G)]$ (with $Y=\emptyset$), and composing the resulting embedding with $g$, finishes the proof.
\end{proof}

In fact, $\soc$ has further nice properties. The following one will be important in Section~\ref{sec:ra56}.
\begin{definition}\label{defn:extension_property}
A signed graph $(G,\sigma)$ \emph{has the $3$-extension property} if for every $\overbar{K_3}$-free graph $H$ with $\lvert V(H)\rvert \leq 3$, every anti-triangle-signing $\tau$ of $H$, every signed induced subgraph $(H',\tau')$ of $(H,\tau)$, and every sign-preserving embedding $f\colon (H',\tau') \to (G,\sigma)$,
there is a sign-preserving embedding 
$g\colon (H,\tau) \to (G,\sigma)$ 
such that $g|_{V(H')} = f$.
\end{definition}
Note that $\tau$ being an anti-triangle-signing is equivalent to saying that if $H$ is a triangle, then $\tau(H) = 0$.

\begin{lemma}\label{lem:extension}
$(\oC,\soc)$ has the $3$-extension property.
\end{lemma}
\begin{proof}
All auxiliary vertices chosen below can be taken outside $P$ (this is possible by density) which makes it possible to apply Lemma~\ref{lem:cswitching}. Fix $(H,\tau)$, $(H',\tau')$, and $f\colon (H',\tau') \to (\oC,\soc)$ as in Definition~\ref{defn:extension_property}. Without loss of generality, we can assume that $f$ is the identity and that $V(H)\setminus V(H') = \{t\}$ for some $t$. If $V(H') = \emptyset$, then we can put $g(t)$ to be an arbitrary vertex of $\oC$.

Next assume that $V(H') = \{x\}$ for some $x$. If $xt$ is not an edge of $H$, then we can put $g(t)$ to be any vertex of $\oC$ such that $\lvert \arg(g(t)/x)\rvert > \frac{2\pi}{3}$. If $xt$ is an edge of $H$, let $v\in V(\oC)\setminus P$ be any vertex of $\oC$ with $\lvert \arg(v/x)\rvert < \frac{2\pi}{3}$. If $\soc(xv) = \tau(xt)$, then we can put $g(t) = v$, otherwise we can use Lemma~\ref{lem:cswitching} with $X = \{x,v\}$ and $S = \{v\}$ to get the desired embedding.

So $V(H') = \{x,y\}$ for some $x$ and $y$. As $H$ contains at least one edge, it is easy to see that there is $v'\in V(\oC)\setminus P$ such that the map $g'$ with $g'(x) = x$, $g'(y) = y$, and $g'(t) = v'$ is an embedding $H\to \oC$. If $g'$ is a sign-preserving embedding, then we are done. So assume that $g'$ is not sign-preserving.
If at most one of $xv'$ and $yv'$ is an edge of $\oC$, then we can use Lemma~\ref{lem:cswitching} with $X = \{x,y,v'\}$ and $S=\{v'\}$ to get the desired embedding. So, we consider the case when both $xv'$ and $yv'$ are edges of $\oC$.
To conclude the proof, we distinguish two cases depending on whether $xy\in E(\oC)$.
If $xy$ is an edge of $\oC$, then we know that 
$$0 = \soc(xy) + \soc(xv') + \soc(yv') = \tau(xy) + \tau(xt) + \tau(yt),$$
which together with $\soc(xy) = \tau(xy)$ implies that 
$$\soc(xv') + \soc(yv') = \tau(xt) + \tau(yt).$$
Consequently, we can again use Lemma~\ref{lem:cswitching} with $X = \{x,y,v'\}$ and $S=\{v'\}$ to get the desired embedding.

So $xy$ is not an edge of $\oC$. Choose $z,w\notin P$ such that $zw\notin E(\oC)$ and $\oC$ induces a $4$-cycle on $\{x,y,z,w\}$. Hence 
$$\soc(xz) + \soc(yz) + \soc(xw) + \soc(yw) = 1,$$
or equivalently,
$$\soc(xz) + \soc(yz) \neq \soc(xw) + \soc(yw).$$
Thus there is $q\in\{z,w\}$ such that
$$\soc(xq)+\soc(yq)=\tau(xt)+\tau(yt).$$
If the two ordered pairs of signs agree, put $g(t)=q$. Otherwise both
signs are reversed, and Lemma~\ref{lem:cswitching}, applied with
$X=\{x,y,q\}$ and $S=\{q\}$, gives the desired embedding.
\end{proof}

\section{Relation algebras}
\label{sec:ra}
The original motivation for this paper comes from the field of relation algebras, which are certain expansions of Boolean algebras that, in a sense, capture the behaviour of certain structures in binary languages. We remark that (integral) relation algebras have also been discovered and studied in a different formalism under the name \emph{hypergroups}~\cite{Hypergroups} (also called \emph{polygroups}~\cite{ComerChromaticPolygroups}). In this section we only introduce the basic definitions; for more details and examples, see e.g.~\cite{Maddux2006-dp,Bodirsky2025RA}.

\begin{definition}\label{def:relation_algebra}
    A \emph{relation algebra} is an algebra $\bA$ with domain $A$ and signature $\{\sqcup, \bar{\phantom{o}}, \bot,\allowbreak \top,\id, \breve{\phantom{o}}, \circ \}$ such that
    \begin{enumerate}
        \item the structure $(A; \sqcup, \sqcap, \bar{\phantom{o}}, \bot,\top)$, with $\sqcap$ defined by $x\sqcap y := \overline{(\bar{x}\sqcup \bar{y})}$, is a Boolean algebra;
        \item $\circ$ is an associative binary operation on $A$, called \emph{composition};
        \item for all $a,b,c \in A$: $(a\sqcup b)\circ c= (a\circ c) \sqcup (b\circ c)$;
        \item \label{itemfour} for all $a\in A$: $a\circ \id =a$;
        \item for all $a\in A$: $\Breve{\Breve{a}} = a$;
        \item for all $a,b\in A$: $(a \sqcup b){\Breve{\phantom{o}}} =\Breve{a}\sqcup \Breve {b}$;
        \item for all $a,b\in A$: $(a \circ b){\Breve{\phantom{o}}}=\Breve{b} \circ \Breve{a}$;
        \item for all $a,b \in A$: $\bar{b} \sqcup \left(\Breve{a} \circ  \overline{(a\circ b)  }\right) =\bar{b}$. 
    \end{enumerate}
 \end{definition}

If $\bA$ is a relation algebra with elements $a$ and $b$, then we write $a \leq b$ if $a \sqcap b = a$ holds in $\bA$. 
 Clearly, $\leq$ defines a partial order on $A$. An element $b \in A \setminus \{\bot^\bA\}$ is called an \emph{atom} if there is no element $a \in A \setminus \{\bot^\bA, b\}$ with $a \leq b$. The set of all atoms is denoted by $A_0$.

Let $\bA$ be a relation algebra. An element $a\in A$ is called \emph{symmetric} if $\Breve{a} = a$, and $\bA$ is called \emph{symmetric} if every element of $\bA$ is symmetric.
For a finite relation algebra $\bA$, the operation $\circ$ is completely determined by its restriction to the atoms. A tuple $(x,y,z) \in (A_0)^3$ is called an \emph{allowed triple} if $z \leq x \circ y$. Otherwise, $(x,y,z)$ is called a \emph{forbidden triple}. One can show that if $(x, y, z)$ is an allowed triple, then also $(\breve{x},z,y)$, $(z,\breve{y},x)$, $(\breve{z},x,\breve{y})$, $(y, \breve{z}, \breve{x})$, and $(\breve{y}, \breve{x}, \breve{z})$ are (see \cite[Theorem 294]{Maddux2006-dp}). In particular, for symmetric relation algebras, the allowed triples are invariant under permutations.

If $R_1,R_2 \subseteq B^2$ are two binary relations, then
$R_1 \circ R_2$ denotes the \emph{composition}
\begin{align*}
R_1 \circ R_2 := \{(x,z) \in B^2 \mid \text{ there exists } y \in B \text{ with } (x,y) \in R_1 \text{ and } (y,z) \in R_2\}.
\end{align*}
Relation algebras are a useful formalism for capturing how certain sets of binary relations on the same domain interact with each other:
\begin{definition}\label{defn:repr}
Let $\bA = (A; \sqcup, \bar{\phantom{o}}, \bot,\top,\id,  \Breve{\phantom{o}}, \circ)$ be a relation algebra. A structure $\fB$ with signature $A$
is called a \emph{representation of $\bA$} if  
\begin{enumerate}
    \item\label{ax:repr:1} $\bot^{\fB} = \emptyset$; 
    \item $\top^{\fB} = \bigcup_{a \in A} a^{\fB}$; 
    \item $\id^{\fB} = \{(u,u) \mid u \in B\}$;
    \item for all $a \in A$ we have $({\overline{a}})^{\fB} = \top^{\fB} \setminus a^{\fB}$;
    \item for all $a \in A$ we have $(\breve{a})^{\fB} = \{(u,v) \mid (v,u) \in a^{\fB} \}$;
    \item\label{ax:repr:6} for all $a,b \in A$ we have $a^{\fB} \cup b^{\fB} = (a \sqcup b)^{\fB}$; 
    \item\label{ax:repr:7} for all $a,b \in A$ we have $a^{\fB} \circ b^{\fB} = (a \circ b)^{\fB}$.
\end{enumerate}
Relation algebras that have a representation are called \emph{representable}.

We say that $\fB$ is \emph{square} if $\top^{\fB} = B^2$ (or, in other words, every pair of vertices is contained in some relation); and it is \emph{finitely bounded} if there is a finite set $\mathcal F$ of finite structures with signature $A$ such that a finite structure $\mathfrak C$ with signature $A$ embeds into $\fB$ if and only if no member of $\mathcal F$ embeds into $\mathfrak C$. 
\end{definition}

\begin{definition}
If $\bA$ is a relation algebra, then an \emph{$\bA$-network} $(V,f)$ consists of a finite set of variables $V$ and a function $f \colon V^2 \to A$  (see, e.g., \cite[Section 1.5.3]{BodirskyBook}). If $\fB$ is a representation of $\bA$, then $(V,f)$ is called \emph{satisfiable in $\fB$} if there exists a function $s \colon V \to B$ such that for all $x,y \in V$ we have 
$(s(x),s(y)) \in f(x,y)^{\fB}$. 
An $\bA$-network $(V,f)$ is called 
\begin{itemize}
    \item \emph{atomic} if for all $x,y \in V$ we have that $f(x,y)$ is an atom in $\bA$;
    \item \emph{consistent} if 
     for all $x,y,z \in V$ we have
     \begin{align*}
         f(x,y) \leq f(x,z) \circ f(z,y) \text{ and } f(x,x) \leq \id;
     \end{align*}
     \item \emph{satisfiable} if it is satisfiable in some representation of $\bA$.
\end{itemize}
\end{definition}

\begin{definition}
The \emph{network satisfaction problem} for a fixed finite relation algebra $\bA$, denoted by $\NSP(\bA)$, is the following computational problem: 
The input consists of an $\bA$-network $(V,f)$. The task is to decide whether $(V,f)$ is satisfiable.

A representation $\fB$ of a relation algebra $\bA$ is called \emph{universal} if every satisfiable 
$\bA$-network is satisfiable in $\fB$.
\end{definition}

\section{The Relation Algebra \texorpdfstring{$56_{65}$}{56\textunderscore 65}}\label{sec:ra56}
In~\cite{Bodirsky2025RA,Bodirsky2025RALIPICS}, Bodirsky, Jahn, Kn\"auer, Konečný, and Winkler systematically studied network satisfaction problems and universal representations for relation algebras with at most four atoms. Out of over a hundred examples, there was only one algebra (called $\ra{56}{65}$ following the terminology of Maddux~\cite{Maddux2006-dp}) for which containment of its $\NSP$ in \NP{} and the existence of a finitely bounded universal representation were left open. In this section we introduce the relation algebra $\ra{56}{65}$ and prove Theorem~\ref{thm:main_ra}, thereby answering~\cite[Question~4.29]{Bodirsky2025RAv1} and concluding the classification of universal representations, and of the computational complexity of $\NSP$ for relation algebras with at most four atoms.

\begin{definition}[The relation algebra $\ra{56}{65}$]
The relation algebra $\ra{56}{65}$ is the unique symmetric relation algebra $\bA$ with four atoms $A_0 = \{\id, N, 0, 1\}$ and the following set of forbidden triples:
\begin{align*}
    &\{ (N,N,N), (1,1,1), (0,0,1), (0,1,0), (1,0,0)\}  \\
    \cup \; &\{(\id, X, Y), (X, \id, Y), (X, Y, \id) : X\neq Y \in A_0\}.
\end{align*}
(Note that the second row only ensures the expected behaviour of the identity atom $\id$.)
\end{definition}

The following fact, which follows by unwinding the definitions, explains our choice to name the atoms $N$, $0$, and $1$. We say that a signed graph $(G,\sigma)$ is \emph{consistent} if $G$ is $\overbar{K_3}$-free and for every triangle $T$ in $G$ it holds that $\sigma(T) = 0$.
\begin{fact}\label{fact:repr}\leavevmode
\begin{enumerate}
    \item \label{fact:repr:1}
Let $\fB$ be a square representation of $\ra{56}{65}$. Define a signed graph $(G,\sigma)$ with $V(G) = B$ such that $\{x,y\} \in E(G)$ if and only if $(x,y) \in 0^\fB \cup 1^\fB$, and $\sigma(\{x,y\}) = 0$ if and only if $(x,y)\in 0^\fB$. Then $(G,\sigma)$ is consistent and has the following two properties:
    \begin{enumerate}
        \item\label{fact:repr:cond1} $G$ contains a non-edge, as well as edges $e_0$ and $e_1$ with $\sigma(e_i) = i$ for $i\in \{0,1\}$.
        \item\label{fact:repr:cond2} $(G,\sigma)$ has the $3$-extension property.
    \end{enumerate}
    \item \label{fact:repr:2}
Conversely, given a consistent graph $G$ with signing $\sigma$ satisfying (\ref{fact:repr:cond1}) and (\ref{fact:repr:cond2}), one can define a structure $\fB$ with domain $V(G)$ and signature the elements of $\ra{56}{65}$ 
and the relations are defined as follows:
    \begin{enumerate}
        \item $(x,y)\in N^\fB$ if and only if $x\neq y$ and $\{x,y\}$ is not an edge of $G$,
        \item $(x,y)\in 0^\fB$ if and only if $x\neq y$ and $\{x,y\}$ is an edge of $G$ with sign 0, 
        \item $(x,y)\in 1^\fB$ if and only if $x\neq y$ and $\{x,y\}$ is an edge of $G$ with sign 1, and
        \item all the other relations are defined uniquely from $N^\fB$, $0^\fB$, and $1^\fB$ according to axioms (\ref{ax:repr:1})--(\ref{ax:repr:6}) of Definition~\ref{defn:repr}.
    \end{enumerate}
Then $\fB$ is a representation of $\ra{56}{65}$.
\end{enumerate}
\end{fact}
Analogously, there is a correspondence between consistent atomic $\ra{56}{65}$-networks where $\id$ is the equality, and finite consistent signed graphs
(thus justifying the name ``consistent'' above). Similarly, we say that a finite signed graph $(G,\sigma)$ is \emph{satisfiable} if the corresponding $\ra{56}{65}$-network is.

\begin{prop}\label{prop:56_balancing}
If $(H,\tau)$ is a satisfiable signed graph, then $\tau$ is an anti-triangle-signing.
\end{prop}
\begin{proof}
By definition, $(H,\tau)$ is satisfiable if and only if its corresponding $\ra{56}{65}$-network $(H,f)$ is. This means that there is a square representation $\fB$ of $\ra{56}{65}$ in which $(H,f)$ is satisfiable. Next, use part~(\ref{fact:repr:1}) of Fact~\ref{fact:repr}
to construct a signed graph $(G,\sigma)$ from $\fB$. It follows that there is a sign-preserving embedding from $(H,\tau)$ into $(G,\sigma)$; without loss of generality, we may assume that the identity is such an embedding. In other words, 
we assume that $H$ is an induced subgraph of $G$ and $\sigma|_H = \tau$. It thus only remains to compute the signs of chordless cycles in $H$. Let $C$ be some chordless cycle in $H$. If $\lvert V(C)\rvert = 3$, then we know that $\sigma(C) = 0$ by consistency. Observe that $H$ contains no chordless cycles on more than 5 vertices as these contain $\overbar{K_3}$ as an induced subgraph, so $\lvert V(C)\rvert \leq 5$.

\medskip

If $\lvert V(C)\rvert = 4$, put $V(C) = \{a,b,c,d\}$ (such that the edges go in this cyclic order). Let $(H',\rho)$ be the signed graph such that $V(H') = \{a,c,t\}$, $E(H') = \{at, ct\}$, and $\rho(at) = \sigma(ab)$ and $\rho(ct) = 1 + \sigma(bc)$. By the $3$-extension property of $(G,\sigma)$, there is a vertex $e\in V(G)$ with $\sigma(ae) = \sigma(ab)$ and $\sigma(ce) = 1 + \sigma(bc)$. If $e=d$, then $$\sigma(C)=\sigma(ab)+\sigma(bc)+\sigma(cd)+\sigma(da)=1,$$ and we are done. So we can assume that $e\neq d$.

Now, if $be$ is an edge of $G$, then by consistency we know that 
$$\sigma(be) = \sigma(ab) + \sigma(ae) + \sigma(be) = 0,$$
but at the same time 
$$1 + \sigma(be) = \sigma(bc) + \sigma(be) + \sigma(ce) = 0,$$
a contradiction. So $be$ is not an edge of $G$. But then $de$ is an edge of $G$ (as otherwise $\{b,d,e\}$ would induce $\overbar{K_3}$), and thus
$$\sigma(cd) + \sigma(ce) + \sigma(de) = 0,$$
and
$$\sigma(ad) + \sigma(ae) + \sigma(de) = 0.$$
Summing these up and using
$\sigma(ae) = \sigma(ab)$ and $\sigma(ce) = 1 + \sigma(bc)$,
we obtain $\sigma(C) = 1$.

\medskip

Finally, if $\lvert V(C)\rvert = 5$, put $V(C) = \{a,b,c,d,e\}$ (such that the edges go in this cyclic order). As above, we can obtain a vertex $f\in V(G)$ with $\sigma(af) = \sigma(ab)$ and $\sigma(cf) = 1+\sigma(bc)$.

As above, we infer that $bf$ is not an edge of $G$, and consequently both $df$ and $ef$ are edges of $G$. By consistency we have the following three equalities:
\begin{align*}
    0 &= \sigma(af) + \sigma(ae) + \sigma(ef)\\
        &= \sigma(ab) + \sigma(ae) + \sigma(ef),\\
    0 &= \sigma(cf) + \sigma(cd) + \sigma(df) \\
        &= 1+\sigma(bc) + \sigma(cd) + \sigma(df), \text{ and}\\
    0 &= \sigma(de) + \sigma(df) + \sigma(ef).
\end{align*}
Summing them up, we get $\sigma(C) = 1$.
\end{proof}

It is easy to see that $(\oC, \soc)$ is consistent and satisfies property~(\ref{fact:repr:cond1}) of Fact~\ref{fact:repr}. Lemma~\ref{lem:extension} says that it also satisfies property~(\ref{fact:repr:cond2}). This means that we can finally prove Theorem~\ref{thm:main_ra}.

\begin{proof}[Proof of Theorem~\ref{thm:main_ra}]
Let $\fB$ be the representation obtained from $(\oC,\soc)$ using part~\ref{fact:repr:2} of Fact~\ref{fact:repr}. Clearly, it is square. It is also finitely bounded, as $\oC$ is finitely bounded (see Theorem~12 and Section~7 of~\cite{Bodirsky2024Circular}) and in addition to these bounds, one only has to forbid cycles of lengths 3, 4, and 5 with signings that are not anti-triangle-signings, and add finitely many forbidden structures ensuring that the representation is square and axioms (\ref{ax:repr:1})--(\ref{ax:repr:6}) of Definition~\ref{defn:repr} are satisfied.

In order to see that it is universal, let $(V,f)$ be a satisfiable $\ra{56}{65}$-network and let $s$ be a satisfying map for $(V,f)$ in some representation. Passing to the image $W=s[V]$ and labelling every pair of points by the unique atom containing it, we obtain a consistent atomic network $(W,g)$ in which $\id$ is precisely equality and $g(s(x),s(y))\leq f(x,y)$ for all $x,y\in V$.

This means that $(W,g)$ corresponds to some satisfiable signed graph $(H,\tau)$, and by Proposition~\ref{prop:56_balancing} we know that $\tau$ is an anti-triangle-signing. Hence, by
Theorem~\ref{thm:main}, $(W,g)$ is satisfiable in $(\oC,\soc)$. Composing a satisfying map for $(W,g)$ with $s$ gives a satisfying map for the original network $(V,f)$.

The ``consequently'' part is exactly the content of~\cite[Lemma~3.5]{Bodirsky2025RA}.
\end{proof}
We remark that a (non-universal) representation of $\ra{56}{65}$ has been found by Lukács; see~\cite[Section~4.7.3]{Bodirsky2025RA}.

\section{Conclusion}
A (countable) structure is \emph{homogeneous} if every isomorphism between finite substructures extends to an automorphism of the whole structure.\footnote{For us, all structures will be countable, that is, finite or countably infinite.} Homogeneous structures are one of the cornerstones of model theory; see e.g.~\cite{Macpherson2011}. Given structures $A$ and $B$ on the same domain such that we can obtain $A$ from $B$ by forgetting some relations, we say that:
\begin{itemize}
    \item $A$ is a \emph{reduct} of $B$, and $B$ is an \emph{expansion} of $A$, if we can obtain $A$ from $B$ by forgetting some relations.
    \item $B$ is a \emph{first-order expansion of $A$} if $B$ is an expansion of $A$ such that all relations of $B$ are first-order definable in $A$; and
    \item $A$ is a \emph{first-order reduct of $B$} if $A$ is a reduct of a first-order expansion of $B$.
\end{itemize}
Following Covington~\cite{Covington} (see also~\cite{Macpherson2011}), we say that 
$B$ is a \emph{homogenization of $A$} if $B$ is homogeneous and has a finite relational signature and $A$ and $B$ are first-order reducts of each other. An $\omega$-categorical structure which has a homogenization is called \emph{homogenizable}.

In~\cite{Bodirsky2024Circular}, Bodirsky and Guzm\'an-Pro proved that $\C$ is homogenizable by adding the quaternary separation relation on the unit circle and a local variant of the (ternary) betweenness relation. One can further add constants for members of $P$
and a unary relation for $C_0$ (see Definition~\ref{defn:sigmac}), so that $\soc$ is then first-order definable in the expanded structure, and 
therefore $(\oC,\soc)$ has an expansion in a finite relational signature which is homogeneous.
However, this expansion is likely not first-order (for example, it is unlikely that $\Aut(\oC,\soc)$ fixes $P$ pointwise), and we thus ask:
\begin{question}\label{q:homogenization}
For which choices of $P$ is $(\oC,\soc)$ homogenizable?
\end{question}
It follows from the homogenization of $\C$ from~\cite{Bodirsky2024Circular} and the fact that $C_0$ and $C_1$ are dense that when solving Question~\ref{q:homogenization}, one only has to consider which members of $P$ belong to $C_0$ and $C_1$ respectively, and what the enumeration of $P$ is (see Definition~\ref{defn:sigmac}).

Once homogenizability is known, the following question becomes interesting:
\begin{question}
What are the optimal Ramsey expansions of the homogenizations of $(\oC,\soc)$ for different choices of $P$?
\end{question}
See for example the recent survey~\cite{Hubicka2025survey} for definitions and a review of structural Ramsey theory.

\medskip

As we have seen above, 
both $(\oC, \soc)$ and the corresponding representation $\fB$ of the relation algebra $\ra{56}{65}$ are reducts of a finitely bounded homogeneous structure,  and 
consequently, they fall into the scope of the Bodirsky--Pinsker conjecture in the area of infinite-domain constraint satisfaction problems~\cite{BPP-projective-homomorphisms}.

\begin{theorem}
    The representation $\fB$ constructed in the proof of Theorem~\ref{thm:main_ra} pp-constructs $K_3$.
\end{theorem}
\begin{proof}
    Observe that $(B, N^\fB)$ is isomorphic to $\C$.
    The rest follows from~\cite[Theorem~45]{Bodirsky2024Circular}, which gives a pp-construction of $K_3$ in $\C$.
\end{proof}

It turns out that there is a pp-definition of $N^\fB$ using only the relations $0^\fB$ and $1^\fB$, and consequently, $\mathrm{CSP}(\oC,\soc)$ is also NP-complete. (Here $(\oC,\soc)$ is seen as a 2-edge-coloured graph, or equivalently, as a structure with two binary relations corresponding to edges with sign 0 and 1 respectively.)

It would be interesting to know whether one needs both labels for hardness. More formally, let $G_0$ be the graph whose edges are exactly the edges of $(\oC,\soc)$ with sign $0$, and define $G_1$ analogously. Clearly, $G_0$ contains an infinite clique, and hence $\mathrm{CSP}(G_0)$ is trivial. On the other hand, the complexity of $\mathrm{CSP}(G_1)$ remains open.

\begin{question}\label{q:cspgraph}
    What is the complexity of $\mathrm{CSP}(G_1)$?
\end{question}

Both $G_0$ and $G_1$ are natural graphs. It would also be interesting to see if there is a simple description of finite graphs which embed into $G_0$ and $G_1$ respectively.

\section*{Disclosure of AI usage}
We used ChatGPT (GPT-5.6) for checking the manuscript for mathematical and presentational mistakes. During this process, it also found a pp-definition of $N^\fB$ in $(\oC,\soc)$, thereby answering a question posed in an earlier preprint version of this paper.

\section*{Acknowledgments}
We are grateful to Tom Zaslavsky for generously sharing his deep knowledge of signed graphs and pointing us towards Truemper's theorem.

\bibliography{56}

\end{document}